\documentclass{amsart}
\usepackage{amsmath,amscd,amsfonts,amssymb,latexsym,amsthm,array}

\usepackage{hyperref}
\theoremstyle{plain}
\newtheorem{thm}{Theorem}
\newtheorem{prop}[thm]{Proposition}
\newtheorem{lemma}[thm]{Lemma}

\newtheorem{cor}[thm]{Corollary}

\theoremstyle{definition}

\theoremstyle{remark}

\DeclareMathOperator{\supp}{supp}

\newcommand{\N}{\mathbb{N}}


\allowdisplaybreaks

\begin{document}

\title{Duals of Tirilman spaces have unique subsymmetric basic sequences}

%\author{S. J. Dilworth, D. Kutzarova, B. Sari and S. Stankov}

\author{S. J. Dilworth}
\address{Department of Mathematics, University of South Carolina, Columbia, SC
29208, USA}
\email{dilworth@math.sc.edu}

\author{D. Kutzarova}
\address{Department of Mathematics, University of Illinois Urbana-Champaign, Urbana, IL 61807, USA; Institute of Mathematics and Informatics, Bulgarian
Academy of Sciences, Sofia, Bulgaria}
\email{denka@illinois.edu}

\author{B. Sar\i}
\address{Department of Mathematics, University of North Texas 1155 Union Circle 311430
Denton, Texas 76203-5017, USA}
\email{bunyamin.sari@unt.edu}

\author{S. Stankov}
\email{erejnion@gmail.com}
\thanks{The first author was supported by Simons Foundation Collaboration Grant No. 849142. The second author was
supported by Simons Foundation Collaboration Grant No. 636954. The first three
authors thank the Workshop in Analysis and Probability at
Texas A\&M University (2022) for support where this work was initiated.}
\maketitle

\begin{abstract}
The Tirilman spaces $Ti(p,\gamma)$, $1<p<\infty$, were introduced by Casazza and Shura as variations of the spaces constructed by Tzafriri. We prove that all subsymmetric basic sequences in the dual space $Ti^*(p,\gamma)$ are equivalent to its  canonical subsymmetic but not symmetric basis.
\end{abstract}

\section{Introduction}

Symmetric structures play an important role in the theory of Banach spaces. A basic sequence $(x_j)_{j=1}^\infty$ is symmetric if the rearranged sequence $(x_{\pi(j)})_{j=1}^\infty$ is equivalent to  $(x_j)_{j=1}^\infty$ for any permutation $\pi$ of $\N$. Recall that a sequence $(x_j)_{j=1}^\infty$ is a basic sequence if it is a (Schauder) basis of its closed linear span; two basic sequences $(x_j)_{j=1}^\infty$ and $(y_j)_{j=1}^\infty$ are said to be equivalent provided a series $\sum_{j=1}^\infty a_jx_j$ converges if and only if $\sum_{j=1}^\infty a_jy_j$ does.

The class of subsymmetric basic sequences, that is, those that are unconditional and equivalent to all of their {\em subsequences} \cite{LT}, is formally more general than the class of symmetric ones. For a while, these two concepts were believed to be equivalent until Garling \cite{G} provided a counterexample. Later, subsymmetric bases became important on their own within the general theory. For instance, the first arbitrarily distortable Schlumprecht space \cite{S} has a subsymmetric basis which is not symmetric.

Albiac, Ansorena and Wallis \cite{AAW} used Garling-type spaces to provide the first example of a Banach with a unique subsymmetric basis which is not symmetric. However, as shown in a sequel paper \cite{AADK}, that space contains a continuum of non-equivalent subsymmetric basic sequences. Altshuler \cite{A} (see also Example 3.b.10 in \cite{LT}) constructed a space which is not isomorphic to $c_0$ or $\ell_p$ for any $1<p<\infty$ and in which all symmetric basic sequences are equivalent to its symmetric basis. Recently, the first example of a Banach space with {\em a unique subsymmetric basic sequence which is not symmetric} is given in \cite{CDKM}.
That answered a question posed in \cite{KMP} and \cite{AADK}. The space under consideration was $Su(T^*)$ \cite{CS}, the subsymmetric version of $T^*$. As it became customary, $T$ is the space considered by Figiel and Johnson \cite{FJ} and its dual $T^*$ is the original space constructed by Tsirelson \cite{T}, the first example of a space which does not contain an isomorphic copy of $c_0$ or $\ell_p$, $1\le p<\infty$.

In this paper we give {\em more examples of spaces with a subsymmetric but not symmetric basis which contain, up to equivalence, a unique subsymmetric basic sequence.} These examples are based on Tzafriri spaces.  Tzafriri \cite{Tz} had constructed (counter)-examples of spaces with (symmetric bases) showing that the notions of equal-norm type $p$ and equal-norm-cotype $q$ are not equivalent to the notions of type $p$ and cotype $q$ for $p,q\neq 2$, respectively. The Tirilman spaces $Ti(p,\gamma)$, where $1<p<\infty$ and $0<\gamma<1$, are modified Tzafriri spaces, which were introduced and studied by Casazza and Shura \cite{CS}. They were named after Tzafriri's Romanian surname. We prove that for $1<p<\infty$ and sufficiently small $0<\gamma<1$, the dual space $Ti^*(p,\gamma)$, whose canonical basis is subsymmetric but not symmetric contains, up to equivalence, a unique subsymmetric basic sequence. That is, all the subsymmetric basic sequences are equivalent to the canonical basis. The method of our proof is parallel to the one in \cite{CDKM}: While there the normalized block bases $(x_j)$ of the canonical basis of $Su(T^*)$ with the property $\|x_j\|_\infty\to 0$ are shown to be asymptotic-$c_0$ sequences, we show that the similar block bases in $Ti^*(p,\gamma)$ yield asymptotic-$\ell_q$ sequences, where $\frac1p+\frac1q=1$. Moreover, unlike its dual {\em $Ti(p,\gamma)$ has continuum many non-equivalent subsymmetric basic sequences.} This follows immediately from Theorem 21 of \cite{CDKM} which states that if a subsymmetric basis $(e_i)$ is not equivalent to the unit vector basis of $c_0$ or $\ell_p$ then either $(e_i)$ or $(e^*_i)$ admits a continuum of non-equivalent subsymmetric block bases. 

\section{Spaces with a unique subsymmetric basic sequence}

Given two basic sequences
$(x_n)_{n=1}^\infty$ and $(y_n)_{n=1}^\infty$ in Banach spaces $X$ and
$Y$, respectively, we say that $(x_n)_{n=1}^\infty$ \emph{$K$-dominates}
$(y_n)_{n=1}^\infty$ if the bounded linear operator  $T(x_n)=y_n$ from
$[(x_n)_{n=1}^\infty]$ to $[(y_n)_{n=1}^\infty]$ has norm $\|T\| \le K$.  We say that  $(x_n)_{n=1}^\infty$ \emph{dominates}  $(y_n)_{n=1}^\infty$ if  $(x_n)_{n=1}^\infty$ $K$-dominates
$(y_n)_{n=1}^\infty$ for some $K<\infty$.  A \emph{block basis} with
respect to a basic sequence $(x_n)_{n=1}^\infty$ is a sequence
$(y_n)_{n=1}^\infty$ of non-zero vectors of the form $y_n=\sum_{k=p_n+1}^{p_{n+1}} a_k x_k$ where $p_1 < p_2 < \cdots$ is an
increasing sequence of natural numbers. For a vector $x$ in the closed linear span of $(x_n)_{n=1}^\infty$, its support (with respect to $(x_n)_{n=1}^\infty$) is the set of indices of its non-zero coefficients. For finite sets of natural numbers $E$ and $F$ we say that $E < F$  if $\max(E)
<\min(F)$. For a natural number $n$, we say $n < x$, resp. $n \le x$, if $n <
\min(\supp(x))$, resp. $n \le  \min(\supp(x))$. A basic sequence $(x_n)$ is called \emph{1-subsymmetric} if it is $1$-unconditional and isometrically equivalent to its subsequences.

A basic sequence $(x_j)_{j=1}^\infty$ is called \emph{(strongly) asymptotic}-$\ell_p$, $1\le p<\infty$ if there exist a constant $C>0$ such that for every $m\in\N$ there is an $M\in \N$ such that for every normalized block basis $(y_j)_{j=1}^m$ of $(x_j)_{j=M}^\infty$ and any set of real numbers $(a_i)$, we have
$$
\frac1C\left(\sum_{i=1}^m |a_i|^p\right)^{\frac1p}\le \left\|\sum_{i=1}^m a_iy_i\right\|\le C\left(\sum_{i=1}^m |a_i|^p\right)^{\frac1p}.
$$
Although we will drop the term `strongly' when referring to asymptotic-$\ell_p$ sequences, it is important to note this is a stronger version of the original definition from \cite{MMT} which was given in a more general setting.

Let $1<p<\infty$ and $0<\gamma<1$. As in the case of Tsirelson space, the norm is defined via an implicit equation. For all $a=(a_i)\in c_{00}$, the
linear space of finitely supported real-valued sequences, define
$$
\|a\|=\max\left\{ \|a\|_\infty,\gamma\sup\frac{\sum_{j=1}^n\|E_ja\|}{n^{\frac1q}}\right\},
$$
where $\frac1p+\frac1q=1$ and the inner supremum is taken over all finite consecutive sets of natural numbers $1\le E_1<\cdots<E_n$ and all $n$. This norm can be computed via the limit of a recursive sequence of norms. We refer to  \cite{CS}, Section X.d.5, for more details.
The Tirilman space $Ti(p,\gamma)$ is the completion of $(c_{00},\|\cdot\|)$. It follows from the definition that the unit vectors $(e_n)_{n=1}^\infty$ form a 1-subsymmetric basis for $Ti(p,\gamma)$. We shall summarize some of their known properties. The first one is the obvious analogue of Proposition X.d.8 \cite{CS} which was proved for $Ti(2,\gamma)$.

\begin{prop}\label{p1}
For every $1<p<\infty$ and $0<\gamma<1$, the canonical basis $(e_n)_{n=1}^\infty$ is $1$-dominated by every normalized block basis of $(e_n)_{n=1}^\infty$.
\end{prop}

Some further properties of $Ti(p,\gamma)$ that were proved in \cite{CS} for $Ti(2,\gamma)$ were listed in Theorem 6.1 \cite{Sa}.

\begin{prop}\label{p2}
Let $1<p<\infty$. Then for sufficiently small $0<\gamma<1$ the  following hold for $Ti(p,\gamma)$.

$(i)$ for any normalized successive blocks $(x_j)_{j=1}^\infty$ of the basis $(e_i)$, we have
$$
\gamma n^{\frac1p}\le\left\|\sum_{j=1}^n x_j\right\|\le 3^{\frac1q}n^{\frac1p}.
$$

$(ii)$ $Ti(p,\gamma)$ does not contain isomorphs of any $\ell_r$, $1\le r<\infty$ or of $c_0$. In particular, $Ti(p,\gamma)$ is reflexive.
\end{prop}

\noindent
\textbf{Remark.}
We shall apply the above proposition for $\gamma<3^{-\frac1q}$.

\medskip

Actually, we need the more general version of the right-hand inequality of $(i)$, which is the $p$-analogue of Lemma X.d.4 \cite{CS}.

\begin{prop}\label{p3}
If $0<\gamma<3^{-\frac1q}$ and $(x_j)_{j=1}^n$ are block vectors in $Ti(p,\gamma)$ with consecutive supports, $n\in\N$, then
$$
\left\|\sum_{j=1}^n x_j\right\|\le 3^{\frac1q}\left(\sum_{j=1}^n\|x_j\|^p\right)^{\frac1p}.
$$
\end{prop}

As an immediate corollary we obtain the following

\begin{lemma}\label{l4}
Let $0<\gamma<3^{-\frac1q}$. Let $(x^*_j)$ be a normalized block basis of $(e_j^*)$ in the dual space $Ti^*(\gamma,p)$. Then for every $n$ and every choice of real numbers $(a_j)_{j=1}^n$, we have
$$
\left\|\sum_{j=1}^n a_jx_j^*\right\|\ge \frac1{3^{\frac1q}}\left(\sum_{j=1}^n|a_j|^q\right)^{\frac1q}.
$$
\end{lemma}

\begin{proof}
For any $1\le j\le n$ choose an $x_j\in Ti(\gamma,p)$ with $\|x_j\|=1$ and $x_j^*(x_j)=1$. Let $(a_j)_{j=1}^n$ be a set of real numbers. By 1-unconditionality we may assume that $a_j\ge 0$ and $\supp x_j\subseteq \supp x^*_j$. Then by duality,
\begin{eqnarray*}
\sum_{j=1}^na_j^q&=& \sum_{j=1}^n a_jx_j^*(a_j^{\frac qp}x_j)=
\left(\sum_{j=1}^n a_j x_j^*\right)\left(\sum_{j=1}^n a_j^{\frac qp}x_j\right)
\\
&\le& \left\|\sum_{j=1}^n a_jx_j^*\right\| \left\|\sum_{j=1}^n a_j^{\frac qp}x_j\right\|
\\
&\le& \left\|\sum_{j=1}^n a_jx_j^*\right\| 3^{\frac1q}\left(\sum_{j=1}^n a_j^q\right)^{\frac1p},
\end{eqnarray*}
which gives the needed inequality.
\end{proof}

\begin{prop}[\cite{Sa}]\label{p5}
Let $1<p<\infty$ and let $\gamma>0$ be sufficiently small. Then $Ti(p,\gamma)$ contains no symmetric basic sequence.
\end{prop}

\noindent
\textbf{Remark.}
It was proved in \cite{JKO} that $c_0$ is finitely representable in $Ti(2,\frac12)$ (disjointly w.r.t. $(e_j)$) which provides an alternative proof that $(e_j)$ is not symmetric.

\medskip

\begin{lemma}\label{l6}
Let $(e_i)$ be a $1$-unconditional basis of a reflexive Banach space
$X$ which is $K$-dominated by its normalized block bases, where $K\ge
1$. Then $(e_i^*)$ $K$-dominates all normalized block bases of $(e_i^*)$ in the dual space $X^*$.
\end{lemma}

\begin{proof}
Let $(x_i^*)$ be a normalized block-basis of $(e_i^*)$ and let $(a_i)_{i=1}^n$, $n\in\N$, be an arbitrary set of real numbers. $(e_i^*)$ is also 1-unconditional, so we may assume that $a_i\ge 0$ for all $1\le i\le n$. Pick a norming element $w\in X$, $\|w\|=1$, $\left(\sum_{i=1}^n a_ix_i^*\right)(w)=\left\|\sum_{i=1}^n a_ix_i^*\right\|$. Denote $A_i=\supp(x_i^*)$.

The 1-unconditionality of $(e_i)$ allows us to assume that
$$
\supp(w)\subseteq \bigcup_{i=1}^n A_i.
$$
Let $w_i=w|_{A_i}$ be the restriction of $w$ to the set $A_i$. Denote $\|w_i\|=c_i$ and $B=\{1\le i\le n\colon c_i\ne 0\}$. By 1-unconditionality, $c_i\le 1$, $1\le i\le n$. For each $i\in B$, let $z_i=\frac{w_i}{c_i}$. Clearly $(z_i)_{i=1}^n$ is a normalized block-basis of $(e_i)_{i=1}^\infty$ and
$$
w=\sum_{i\in B} c_iz_i.
$$
Then,
\begin{eqnarray*}
\left\|\sum_{i=1}^n a_ix_i^*\right\|
&=&
\left(\sum_{i=1}^n a_ix_i^*\right)\left(\sum_{i\in B} c_iz_i\right)\\
&=&
\sum_{i\in B} a_ic_ix_i^*(z_i)\le \sum_{i\in B} a_ic_i\\
&=&
\left(\sum_{i\in B} a_ie_i^*\right)\left(\sum_{i\in B} c_ie_i\right)
\le
\left\|\sum_{i\in B} a_ie_i^*\right\|\cdot
\left\|\sum_{i\in B} c_ie_i\right\|
\end{eqnarray*}
By the $K$-domination,
$$
\left\|\sum_{i\in B} c_ie_i\right\|\le K \left\|\sum_{i\in B} c_iz_i\right\|=K.
$$
Thus,
$$
\left\|\sum_{i=1}^n a_ix_i^*\right\|\le K\left\|\sum_{i\in B} a_ie_i^*\right\| \le K\left\|\sum_{i=1}^n a_ie_i^*\right\|.
$$
\end{proof}

\begin{lemma}\label{l7}
For any $n$ and any sequence of normalized blocks $(x_j^*)_{j=1}^n$ of $(e_j^*)_{j=1}^\infty$ in $Ti^*(p,\gamma)$,
$$
\left\|\sum_{j=1}^n x_j^*\right\|\le \frac{n^{\frac1q}}{\gamma}.
$$
\end{lemma}

\begin{proof}
By the previous Lemma \ref{l6} and Proposition \ref{p1}, $(x_j^*)_{j=1}^n$ is 1-dominated by $(e_j^*)_{j=1}^n$, so
$$
\left\|\sum_{j=1}^n x_j^*\right\|\le \left\|\sum_{j=1}^n e_j^*\right\|.
$$
The vector $\frac\gamma{n^{\frac1q}}\sum_{j=1}^n e_j^*$ belongs to the unit ball of $Ti^*(p,\gamma)$, see e.g. \cite{M}, so $ \left\|\sum_{j=1}^n e_j^*\right\|\le\frac{n^{\frac1q}}\gamma$.
\end{proof}

\begin{lemma}\label{l8}
$Ti^*(p,\gamma)$ does not contain an isomorphic copy of $\ell_q$ $(\frac1p+\frac1q=1)$.
\end{lemma}

\begin{proof}
Assume the contrary. Without loss of generality we may assume that a
normalized block basis $(x_j^*)$ of $(e_j^*)$ is $C$-equivalent to the
unit vector basis of $\ell_q$. Denote $I_j=\supp(x_j^*)$. Choose norming
elements $x_j\in Ti(p,\gamma)$, $\|x_j\|=1$, $x^*_j(x_j)=1$. By the
1-unconditionality we may assume that $\supp(x_j)\subseteq
I_j\subset\N$ for all $j\in\N$. Clearly, $I_1<I_1<\cdots$ and denote by $P_j$ the projection on $I_j$.

Define the projection
$$
P(x^*)=\sum_{j=1}^\infty \langle P_j(x^*),x_j\rangle x_j^*.
$$
Then
\begin{eqnarray*}
\|P(x^*)\|&\le& C\left(\sum_{j=1}^\infty|\langle P_j(x^*),x_j\rangle|^q\right)^{\frac1q}\\
&\le& C\left(\sum_{j=1}^\infty\|P_j(x^*)\|^q\right)^{\frac1q}\stackrel{\rm Lemma~\ref{l4}}{\le} 3^{\frac1q}C\|x^*\|.
\end{eqnarray*}
Thus, the subspace generated by $(x^*_j)_{j=1}^\infty$ is complemented in $Ti^*(p,\gamma)$ which implies that $Ti(p,\gamma)$ contains an isomorphic copy of $\ell_p$, a contradiction.
\end{proof}

By Lemma \ref{l4} and \ref{l7} for all $n$ and all normalized block sequences $(u_i)_{i=1}^n$ in $Ti^*(p,\gamma)$ we have $\|\sum_{i=1}^nu_i\|\stackrel{K}\sim n^{1/q}$ for some $K$. In \cite{JKO} it was shown that spaces with such a property are saturated by asymptotic-$\ell_q$ sequences. An inspection of their proof (of Theorem 3.7) shows that any block sequence $(x_i)$ with $\|x_i\|_{\infty}\to 0$ is asymptotic-$\ell_q$. Thus the next Proposition follows from the proof of Theorem 3.7 in \cite{JKO}. We reproduce the proof for completeness, which is slightly easier in our case.

\begin{prop}\label{l9}
Let $1<p<\infty$ and $0<\gamma<3^{-1/q}$. Every normalized block sequence $(x_i)_{i=1}^\infty$ in $Ti^*(p,\gamma)$ satisfying $\|x_i\|_\infty \to 0$ is an asymptotic $\ell_q$ basic sequence where $\frac1p+\frac1q=1$.
\end{prop}

\begin{proof}

Let $m\in \N, m\ge 2$. Choose $\varepsilon,\delta>0$, and $\delta'$ satisfy 

\begin{equation}\label{ep-delta}
0<\varepsilon<\frac1{4m 3^{1/q}}, \quad  \delta=\frac{\varepsilon}{6\gamma^{-1}m}, \quad 0<\delta'<\frac{\delta^{q+1}}{\gamma^{-q}m}.
\end{equation}

Let $M\in\N$ be such that $\|x_i\|_{\infty}<\delta'$ for all $i\ge M$. Let $(y_i)_{i=1}^m$ be a normalized block basis of $(x_i)_{i\ge M}$. We will show that for all scalars $(a_i)_{i=1}^m$ with $\sum_{i=1}^m |a_i|^q=1$ we have

\begin{equation}\label{equi}
\frac1{3^{1/q}}\le\left\|\sum_{i=1}^m a_iy_i\right\|\le 3^{q+1}\gamma^{-q}.
\end{equation}

Fix $(a_i)_{i=1}^m$. The left hand side inequality holds for all normalized block vectors and was shown in Lemma \ref{l4}.

For each $i$, write $a_iy_i=\sum_{j=1}^{n_i+1}y_{i,j}$ where $y_{i,j}$'s are successive blocks with $\delta\le\|y_{i,j}\|<\delta+\delta'$ and $\|y_{i, n_i+1}\|<\delta$. Then by Lemma \ref{l4}
$$|a_i|=\|a_iy_i\|\ge 3^{-1/q}\Big(\sum_{i=1}^{n_i+1}\|y_{i,j}\|^q\Big)^{1/q}\ge 3^{-1/q} \delta n^{1/q}_i.$$
Thus for all $1\le i\le m$, 
\begin{equation}\label{n_i}
n_i\le \frac{3|a_i|^q}{\delta^q}.
\end{equation}

Moreover, by shrinking each $y_{i,j}$ to have norm exactly $\delta$ at a cost of $\delta'$ we have by Lemma \ref{l7} that
$$\|a_iy_i\|\le  \gamma^{-1}\delta n^{1/q}_i+n_i\delta'+\delta\le \gamma^{-1}\delta n^{1/q}_i+2\delta$$
since $n_i\delta'+\delta\le \frac{3}{\delta^q}\delta'+\delta\le \frac{3}{\delta^q}\frac{\delta^{q+1}}{\gamma^{-q}m}+\delta\le \frac{\delta}{m}+\delta<2\delta$. 

If $|a_i|\ge \varepsilon$ then $n_i\neq 0$ and from above $\varepsilon\le \|a_iy_i\|\le \gamma^{-1}\delta n^{1/q}_i+2\delta\le 3\gamma^{-1}\delta n^{1/q}_i$ since $\gamma^{-1} n^{1/q}_i>1$. Thus
$$n^{1/q}_i>\frac{\varepsilon\gamma}{3\delta}=2m.$$

Let $$N=\sum_{\{i:|a_i|\ge \varepsilon\}}n_i.$$ Then $2m\delta<\delta n^{1/q}_i\le \delta N^{1/q}$, and by above $N\delta'<\delta$. We have, using Lemma \ref{l7} again, 

\begin{align*}
\Big\|\sum_{\{i:|a_i|\ge \varepsilon\}}a_i y_i\Big\|&\le \gamma^{-1}\delta N^{1/q}+N\delta'+m\delta\\
&\le \gamma^{-1}\delta N^{1/q}+\delta+m\delta\\
&\le \gamma^{-1}\delta N^{1/q}+2m\delta\\
&\le \gamma^{-1}\delta N^{1/q}+\delta N^{1/q}.\\
\end{align*} Thus 
\begin{equation}\label{big-coeff}
\Big\|\sum_{\{i:|a_i|\ge \varepsilon\}}a_i y_i\Big\|\le 2\gamma^{-1}\delta N^{1/q}.
\end{equation}

On the other hand, by Lemma \ref{l4} we have
\begin{align}\label{big-coeff-lower}
\Big\|\sum_{\{i:|a_i|\ge \varepsilon\}}a_i y_i\Big\|&\ge  3^{-1/q}\Big(\sum_{\{i:|a_i|\ge \varepsilon\}}|a_i|^q\Big)^{1/q}
\ge  3^{-1/q}(1-\varepsilon m)^{1/q}\ge \frac12 3^{-1/q}, 
\end{align}

and

\begin{align*}
\Big\|\sum_{\{i:|a_i|< \varepsilon\}}a_i y_i\Big\|<m\varepsilon\stackrel{\ref{ep-delta}}<\frac14 3^{-1/q}\stackrel{\ref{big-coeff-lower}}\le\frac12\Big\|\sum_{\{i:|a_i|\ge \varepsilon\}}a_i y_i\Big\|\stackrel{\ref{big-coeff}}<\gamma^{-1}\delta N^{1/q}.
\end{align*}

Thus by the triangle inequality
\begin{align*}
\Big\|\sum_{i=1}^m a_i y_i\Big\|^q&<3^q\gamma^{-q}\delta^q N\\
&\le 3^q\gamma^{-q}\delta^q \sum_{\{i:|a_i|\ge \varepsilon\}}n_i\\
&\le 3^{q+1}\gamma^{-q} \sum_{\{i:|a_i|\ge \varepsilon\}}|a_i|^q\quad \text{by}\ (\ref{n_i})\\
&\le 3^{q+1}\gamma^{-q}.
\end{align*}

\end{proof}

\begin{thm}\label{t10}
Let $1<p<\infty$ and $\gamma>0$ be sufficiently small. Every subsymmetric basic sequence in the dual space $Ti^*(p,\gamma)$ is equivalent to the subsymmetric canonical basis $(e^*_j)_{j=1}^\infty$ which is not symmetric.
\end{thm}

\begin{proof}
By Proposition \ref{p5} $(e^*_j)_{j=1}^\infty$ is not symmetric. 

Let $(x^*_j)_{j=1}^\infty$ be a normalized subsymmetric basic sequence in $Ti^*(p,\gamma)$. By passing to a subsequence we may assume that $(x^*_j)_{j=1}^\infty$ is a block basis of $(e_j)_{j=1}^\infty$. If we suppose that $\lim_{j\to\infty}\|x_j^*\|_\infty=0$, then by combining Lemma \ref{l4}, Lemma \ref{l6} and Proposition \ref{l9}, we obtain that $(x^*_j)_{j=1}^\infty$ is an asymptotic $\ell_q$ basic sequence. Then the subsymmetry would imply that $(x^*_j)_{j=1}^\infty$ is equivalent to the unit vector basis of $\ell_q$ which contradicts Lemma \ref{l8}.

Thus, by passing again to a subsequence, we may assume that for all $j\in\N$, $\|x_j^*\|_{\infty}\ge c$ for some $c>0$. Then $(x^*_j)_{j=1}^\infty$ $c$-dominates $(e^*_j)_{j=1}^\infty$. On the other hand, by Lemma \ref{l6} $(x^*_j)_{j=1}^\infty$ is 1-dominated by $(e^*_j)_{j=1}^\infty$ and therefore, they are equivalent.
\end{proof}
Reflexivity of  $Ti(p,\gamma)$ and duality yield the following
\begin{cor}  Let $1<p<\infty$ and $\gamma>0$ be sufficiently small. Every subsymmetric basis of a  quotient space of  $Ti(p,\gamma)$ is equivalent to the  canonical basis $(e_j)_{j=1}^\infty$.
\end{cor}

\begin{prop}[\cite{CDKM}]\label{p11}
Let $(e^*_i)$ be a subsymmetric basis which is not equivalent to the unit vector basis of $\ell_p$ or $c_0$. Then either $(e_i)$ or $(e_i^*)$ admits a continuum of non-equivalent subsymmetric block bases.
\end{prop}

This, together with Theorem \ref{t10}, give us the following

\begin{cor}\label{c12}
For $1<p<\infty$ and sufficiently small $\gamma$, the basis $(e_i)$ of $Ti(p,\gamma)$ has a continuum many non-equivalent subsymmetric block bases.
\end{cor}

\end{document}